\documentclass[a4paper,11pt]{article}
\usepackage{amssymb}
\usepackage{graphicx}
\usepackage{tikz}
\usetikzlibrary{intersections}
\usetikzlibrary{svg.path}
\usetikzlibrary{decorations.pathmorphing}

\usepackage{pgfplots}
\pgfplotsset{compat=1.11}
\usepackage{mathrsfs}
\usepackage{amsmath}
\usepackage{amsthm}
\usepackage{enumerate}
\usepackage{color}
\usepackage{verbatim}\vspace{0.2cm}
\usepackage{todonotes}
\usepackage{graphicx} 
\usepackage{amsmath, relsize} 
\usepackage{todonotes} 
\usepackage{amsthm} 
\usepackage{amsfonts}
\usepackage{amssymb}
\usepackage{array}
\usepackage{url}
\usepackage{amsmath}
\usepackage{amssymb}
\usepackage{amsthm}
\usepackage{bbm}

\usepackage{pgfplots}
\pgfplotsset{compat=1.11}
\usepackage{mathrsfs}
\usepackage{amsmath}
\usepackage{amsthm}
\usepackage{enumerate}
\usepackage{color}
\usepackage{verbatim}
\usepackage{todonotes}

\usepackage[margin=3cm, includefoot, footskip=30pt]{geometry}
\usepackage{array}
\newcolumntype{e}{>{\displaystyle}r @{\,} >{\displaystyle}c @{\,} >{\displaystyle}l}
\usepackage{amsfonts,amssymb,amsopn,amscd}

\usepackage[colorlinks]{hyperref}
\hypersetup{final}
\usepackage{xcolor}
\hypersetup{
    colorlinks,
    linkcolor={blue!40!black},
    citecolor={blue!40!black},
    urlcolor={blue!80!black}
}
\usepackage{nameref}

\usepackage[figure,table]{hypcap} % Correct a problem with hyperref
\hypersetup{
   bookmarksnumbered,
   pdfstartview={FitH},
   citecolor={blue},
   linkcolor={blue},
   urlcolor={black},
   pdfpagemode={UseOutlines}
}

\theoremstyle{plain}
\newtheorem{Theorem}{Theorem}[section]

\newtheorem{Lemma}[Theorem]{Lemma}
\newtheorem{Proposition}[Theorem]{Proposition}

\theoremstyle{definition}

\newtheorem{Definition}[Theorem]{Definition}
\newtheorem{Remark}[Theorem]{Remark}

\newcounter{Condition}

\renewcommand{\P}{\mathbb{P}}
\newcommand{\E}{\mathbb{E}}

\newcommand{\N}{\mathbb{N}}

\newcommand{\Z}{\mathbb{Z}}

\renewcommand{\d}{\mathbf{d}}
\newcommand{\GVis}{\operatorname{w}}
\newcommand{\CorG}{{\widetilde{m}}}
\newcommand{\GSt}{\operatorname{H}}

\numberwithin{equation}{section}

  \newcounter{constant}

\usepackage{calc}
\def\arraypar#1{\parbox[c]{\textwidth - 2cm}{\centering #1}}
\usepackage{environ}
\NewEnviron{display}{
\begin{equation}\begin{array}{c}
  \arraypar{\BODY}
\end{array}\end{equation}
}
\begin{document}

\title{Abelian oil and water dynamics does not have an absorbing-state phase transition}

\author{Elisabetta Candellero\footnote{ecandellero@mat.uniroma3.it;  Universit\`a Roma Tre, Dip.\ di Matematica e Fisica, Largo S.\ Murialdo 1, 00146, Rome, Italy.} \and Alexandre Stauffer\footnote{a.stauffer@bath.ac.uk;  University of Bath, Dept of Mathematical Sciences, BA2 7AY Bath, UK.} \and Lorenzo Taggi\footnote{lorenzo.taggi@gmail.com;  University of Bath, Dept of Mathematical Sciences, BA2 7AY Bath, UK.}}

\maketitle

\begin{abstract}
The \emph{oil and water} model is an interacting particle system with two types of particles and a dynamics that conserves the number of particles, which belongs to the so-called class of \emph{Abelian networks}. 
Widely studied processes in this class are \emph{sandpiles models} and \emph{activated random walks}, which are known 
(at least for some choice of the underlying graph) to undergo an \emph{absorbing-state phase transition}. This phase transition characterizes the existence of two regimes, depending on the particle density:
a regime of \emph{fixation} at low densities, where the dynamics converges towards an absorbing state and each particle jumps only finitely many times, 
and a regime of \emph{activity} at large densities, where particles jump infinitely often and 
activity is sustained indefinitely. 
In this work we show that the oil and water model is substantially different than sandpiles models and activated random walks, in the sense that it does not undergo an
absorbing-state phase transition and is in the regime of fixation at \emph{all} densities. 
Our result works in great generality: for \emph{any} graph that is vertex transitive and for a large class of initial configurations.
% {\color{red}
% Check abstract.}
% We consider the behavior of a two-type internal aggregation model defined on an infinite vertex-transitive graph $G$.
% In this system there are two types of particles, called \emph{oil} and \emph{water}, which interact according to the following rules.
% When at a vertex $x$ there is at least one oil  and one water, then $x$ ``\emph{fires}'' the pair, that is, each of the two particles takes a step according to simple random walk independently of the other.
% A vertex hosting only particles of the same type (either oil or water) does not fire.
% We show that when we start with an arbitrarily large density of oil-water pairs at each vertex of $G$, the process eventually \emph{fixates} almost surely, meaning that every vertex in any finite set $A\subset V(G)$ will eventually be occupied by \emph{at most} one type of particles.
\end{abstract}

\section{Introduction}
\label{sect:introduction}
We consider an interacting particle system called \emph{oil and water}, which is defined as follows. 
There are two types of particles, which we call \emph{oils} and \emph{waters}. 
Take $G=(V(G),E(G))$ to be an infinite graph, 
and let $\nu$ be a probability measure on the set of non-negative integers $\N$.
The initial configuration of particles is distributed as a product of independent random variables distributed as $\nu$; that is, 
at each vertex $x\in V(G)$ place a random number of oils and independently a random number of waters, both values sampled from the distribution $\nu$.
We denote by $\mu=\mu(\nu)$ the expected number of particles at a given vertex; thus $\mu/2$ is the expectation of a random variable distributed as $\nu$.
We shall refer to this initial configuration as \emph{oil and water at density }$\mu$.

Starting from the above configuration, particles move according to the following dynamics. 
Each vertex of $G$ has an independent Poisson clock of rate $1$. Whenever the clock of a vertex $x$ rings, if $x$ 
hosts at least one oil and one water then it \emph{fires} an oil-water pair: one water and one oil jump independently according to one step of simple random walk on $G$ (that is, each of the two chooses
independently a neighbor of $x$ uniformly at random and jumps there).
On the other hand, if at the time the Poisson clock of $x$ rings, $x$ has no particles or hosts only particles of one type (either oil or water), then $x$ does not fire; in this case we say that $x$ is \emph{stable}.
Note that $x$ may host arbitrarily many particles, but as long as they are all of the same type, $x$ is stable and none of its particles are allowed to jump.
Note also that if we reach a configuration where every vertex of $G$ is stable, which we refer to as a \emph{stable configuration}, then no vertex fires from that time onwards. Thus stable configurations are absorbing states 
for the dynamics. 
%
%\textbf{\color{red}
%Maybe define $\nu$ as any measure on $\mathbb{N}$, and the initial configuration as product of $\nu$, with $\mu$ being the expected value of distribution $\nu$. Then define $\P_\nu$ as before.
%Also, check issue with definition of $\mu$...
%}

\emph{Oil and water} has the so-called \emph{Abelian property}~\cite{BondLevine}, which states that the final configuration of the system does not depend on the order at which vertices fire.
This gives that the times at which the Poisson clocks ring are irrelevant.

There are two possible outcomes of the system: either it is \emph{active}, in which each vertex fires infinitely many times, or it \emph{fixates}, that is each vertex fires finitely many times. 
It is easy to check that if one vertex fires infinitely many times, then all vertices do as well.
We show two fundamental properties in Section~\ref{sect:graphical}. 
The first one is \emph{monotonicity} (Lemma~\ref{lemma:monotonicity}), which gives that if oil and water at density $\mu$ fixates, then 
it also fixates for all densities $\mu'\in[0,\mu]$. 
The second property is a 0-1 law (Lemma~\ref{lemma:01law}), which states that given $\mu$ the probability that the system fixates is either 0 or 1.

With the above two properties, and inspired by results for other models having the Abelian property and also satisfying those properties, 
such as \emph{stochastic sandpiles} and \emph{activated random walks}~\cite{RollaSidoravicius,SidoraviciusTeixeira},
one may conjecture that the oil and water model undergoes a phase transition between \emph{activity} and \emph{fixation} at some critical density $\mu_c$. 
(A more thorough discussion of the relation between oil and water and other models with the Abelian property is given in Section~\ref{sec:relatedModels} below.)

The main result of our paper is to show that the above conjecture is not true, for \emph{any} graph $G$, with the only requirement that $G$ is vertex transitive.
% In fact, our result also holds for a more general initial configuration of particles. 
% 
% In order to state our result in more generality, we introduce some notation.
% We denote an element of the space of possible configurations $\Omega := \mathbb{N}^{V(G)} \times \mathbb{N}^{V(G)}$ by
% \[
%  \eta = \big ( \eta^o(x), \eta^w(x)   \big )_{x \in V(G)},
% \]
% where $\eta^o(x)$  (resp.\ $\eta^w(x)$) corresponds to the number of oils (resp.\ waters) at $x$.
% Given $\mu$, let $\nu$ be a measure on $\Omega$ which is invariant under automorphisms of $G$ and has density $\mu$; thus,
% for any $x\in V(G)$, the expected values of $\eta^o(x)$ and $\eta^w(x)$ are both equal to $\mu$.
% $\P_{\nu}$ denote the probability
% law of the oil and water dynamics when the initial configuration is distributed according to $\nu$.
Let $\P_{\nu}$ denote the probability
law of the oil and water dynamics starting from a configuration of density $\mu=\mu(\nu)$ as above.

\begin{Theorem}\label{thm:fixation}
Let $G$ be an infinite, vertex-transitive graph of finite degree. Then, for any $\nu$ with $\mu=\mu(\nu)<\infty$, 
\[
\P_\nu\bigl[ \text{oil and water fixates}\bigr ]=1.
\]
\end{Theorem}

Oil and water was introduced in \cite{BondLevine} as an example of an Abelian network that is not unary (that is, which has more than one type of particles); see Section~\ref{sec:relatedModels} below. 
The oil and water model was analyzed in \cite{CandelleroGangulyHoffmanLevine} in a different setting. 
They consider the one-dimensional lattice $\mathbb{Z}$, and let the initial configuration be 
given by $N$ oil-water pairs at the origin, with all other vertices initially unoccupied.
Then the oil and water dynamics is run until a stable configuration is obtained; this occurs in finite time, almost surely, since the number of particles is finite.
In this setting, \cite{CandelleroGangulyHoffmanLevine} investigated several statistics of the model, including how long it takes for the process to stop and how far from the origin particles spread as a function of $N$.

\subsection{Related models}\label{sec:relatedModels}
\emph{Oil and water} was introduced in \cite{BondLevine} within the more general framework of Abelian networks, which 
was introduced by Bond and Levine~\cite{BondLevine} building on the work of Dhar on sandpile models~\cite{Dhar}. 
This framework was created with the goal of 
defining a general concept that includes several widely studied processes, 
such as \emph{Abelian and stochastic sandpiles}, \emph{bootstrap percolation}, \emph{rotor-router networks}, \emph{internal DLA} and 
\emph{activated random walks}.
Informally speaking, a particle system (or, more generally, a cellular automaton) is considered an Abelian Network 
if it satisfies the so-called \emph{Abelian property}, which gives that the final configuration of the system does not depend on the order of the interactions.
In other words, the final configuration is invariant to changes in the order at which vertices fire. 

Abelian networks have been widely studied in several disciplines. 
For example, in computer science, they are a fundamental model in distributed systems, as they do not require any central synchronization or shared memory, see~\cite{BondLevine} for more details.
In mathematics and physics, several types of Abelian networks have been investigated, an archetypal example being \emph{sandpile models}~\cite{JaraiSandpile}. 
% One fundamental model in this class is the finite version of the \emph{stochastic sandpile}, which we now define.
% Consider a finite cube in the square lattice $\mathbb{Z}^d$. Given any configuration of particles on the vertices of the cube, we call a vertex stable if 
% it has at most $1$ particle; otherwise it is called unstable. 
% If all vertices are stable, we say that we have a stable configuration. 
% Starting from a stable configuration, we choose a vertex $x$ of the cube uniformly at random, and 
% add a particle to $x$. If $x$ becomes unstable, then it \emph{fires}, meaning that we take $2$ particles from $x$ and each of them independently and uniformly at random choose a neighbor of $x$ on $\mathbb{Z}^d$ and jump there. 
% We keep firing unstable vertices, one by one, and whenever a particle jumps out of the cube, it is immediately removed from the system. 
% Thus, the system eventually reaches a stable configuration. Stochastic sandpile can be shown to satisfy the Abelian property, 
% thus (using a proper representation of the dynamics) the final configuration does not depend on the order at which vertices are fired.
% After reaching a stable configuration, another particle is added to a uniformly random vertex, and the procedure is iterated. This defines a Markov chain in the set of stable configurations, which has a 
% stationary distribution.
% 
The study of sandpile models was initiated in~\cite{BTW1,BTW2} motivated by the observation that they present characteristics of \emph{self-organized criticality}.
This means that as the process evolves, the system drives itself to a ``critical state'' without having to tune any parameter. 
Here, ``critical state'' means that after a long time the configuration shows characteristics that are 
common to systems at criticality. 
Refer to~\cite{JaraiSandpile} for more information about self-organized criticality and sandpile models. 

There have been several works in the physics literature to understand self-organized criticality.  
One approach has been to relate this phenomenon to the more  classical one of phase transitions, called \emph{aborbing-state phase transition} \cite{MDPVZ}. 
This corresponds to a phase transition between a regime of fixation (where for a small density of particles 
the system moves towards an absorbing state) and a regime of activity (where for a large density of particles the activity is sustained indefinitely).
Physicists believe that the presence of an \emph{absorbing-state phase transition} is intrinsically connected 
to the phenomenon of \emph{self-organized criticality}~\cite{MDPVZ}, and even defines a new universality class~\cite{RPV}. 
In particular, physicists studied several systems with a conserved number of particles which are connected to systems from self-organized criticality, and showed non-rigorously that such systems undergo 
an absorbing-state phase transition. Examples of such systems include 
stochastic sandpiles, fixed energy sandpiles, conserved threshold transfer processes, and activated random walks~\cite{MDPVZ,RPV,PV}.

In the mathematics literature, results in this area are much more scarce. Ingenious proofs have been developed to show that stochastic sandpiles and activated random walks
undergo an absorbing-state phase transition in some graphs~\cite{RollaSidoravicius,SidoraviciusTeixeira,StaufferTaggi,BGH,Taggi}, and it is expected that such a result should be true for any vertex-transitive graph.
In this paper, we show that the same is not true for the oil and water model, for any vertex-transitive graph. 
In some sense, the strong interactions between the particles in the oil and water dynamics cause the particles to organize themselves in order to achieve fixation.
To the best of our knowledge, this is the first time that a natural model of an Abelian network (with a conserved number of particles) is shown not to undergo an absorbing-state phase transition.
Another additional feature of our result is that our proof is not egineered for a specific graph, 
but works in any vertex transitive graph and any initial configuration of particles that is obtained from a 
product measure.

\subsection{Proof overview}\label{sec:overview}
Two fundamental properties that will be heavily employed in the proof are the \emph{Abelian property} and the \emph{0-1 law}.
A popular strategy to analyze Abelian networks~\cite{RollaSidoravicius,SidoraviciusTeixeira,BGH,CandelleroGangulyHoffmanLevine}
is to devise a so-called \emph{stabilization algorithm}. For example, if one wants to show fixation (resp., activity), this strategy consists 
of choosing a smart order to fire the vertices, exploiting the Abelian property, in order to obtain that a given vertex \emph{does not fire at all} (resp., fires infinitely many times) with positive probability, 
which by the 0-1 law implies almost surely fixation (resp., activity). 
Usually, the stabilization algorithm exploits the structure of the graph (which, in all the aforementioned papers, was always a grid such as $\mathbb{Z}^d, \, d \geq 1$), making such proofs very much graph dependent. 
Moreover, in some models, such as stochastic sandpiles and activated random walks, where an absorbing-state phase transition takes place, one also uses \emph{monotonicity}; that is, 
it suffices to show fixation for some small enough $\mu$, and to show activity for some large enough $\mu$. 

The oil and water model gives rise to different challenges, since we need to show that the process fixates for \emph{all} $\mu$, no matter how large it may be, and for \emph{all} transitive graphs. 
In order to do this, we had to develop a new proof strategy.
Before describing it, we fix some terminology.
For any vertex $x$, if $x$ has $k_o$ oils and $k_w$ waters, 
we say that $x$ has $k_o \land k_w$ oil-water pairs, where we view each such pair as a matching between an oil particle and a water particle from $x$.
So, each vertex $x$ may only have unpaired particles of at most one type (either oil or water). 

Now suppose that vertex $x\in V(G)$ is unstable, thus $x$ has at least one oil-water pair.
If all neighbors of $x$ have nonzero unpaired oils, when we fire $x$, the water particle that gets to jump from $x$ will be paired to one of the unpaired oils located at the neighbors of $x$ (or to the 
oil particle that jumped from $x$, if both oil and water jump to the same neighbor). 
As a consequence, the number of oil-water pairs in the system does not change. In fact, 
even if the water jumping from $x$ gets paired to a different oil particle, 
we observe that the firing of $x$ effectively causes an oil-water pair to do a step of a \emph{simple random walk} from $x$. 
The same occurs if all neighbors of $x$ have nonzero unpaired waters. 

On the other hand, suppose that $d_w\geq 1$ neighbors of $x$ have unpaired waters, 
$d_o\geq 1$ neighbors of $x$ have unpaired oils, and that
$d_w + d_o = \d$ with $\d$ denoting the degree of each vertex of $G$ 
(that is, each neighbor of $x$ has at least one unpaired particle).
In this case, the number of oil-water pairs changes by either -1, 0 or 1. For example, it changes by $-1$ (resp., $+1$) 
if the water jumps from $x$ to a neighbor with unpaired waters (resp., oils), and the oil jumps from $x$ to a neighbor with unpaired oils (resp., 
waters); in other cases the number of oil-water pairs does not change. We can readily see that 
\[
\text{the number of oil-water pairs changes }
\left \{
\begin{split}
& \text{by $0$ with probability }=1-2\frac{d_w d_o}{\d^2},\\
& \text{by $1$ with probability }=\frac{d_w d_o}{\d^2},\\
& \text{by $-1$ with probability }= \frac{d_w d_o}{\d^2}.
\end{split}
\right .
\]
% where $\d$ denotes the degree of each vertex of $G$.
The above gives that, in this case, the configuration of oil-water pairs behaves as a \emph{critical branching random walk} on $G$.
Suppose now that $x$ has at least one neighbor with no unpaired particles (such neighbors are called holes), then
we have that the configuration of oil-water pairs behaves as a \emph{subcritical} 
branching random walk. 

Putting all these cases together, when a vertex $x$ fires, 
the configuration of oil-water pairs behaves either as a simple random walk, as a critical branching random walk, or as a subcritical branching random walk, 
depending on the environment of unpaired particles at the neighbors of $x$. Moreover, it behaves as a subcritical branching random walk only when $x$ is the neighbor of a hole.

Intuitively, since oil-water pairs cause a vertex to fire, in order to show fixation we need to show that the number of oil-water pairs decreases quickly. 
Thus, we want to show that for a large enough number of steps we fire a vertex that neighbors a hole.

The proof works by contradiction. We assume that the system is active, which implies that each vertex fires a very large number of times.
Now consider a vertex $x$ that fires $k$ times, and let $y$ be a neighbor of $x$ which, for instance, has unpaired oils. 
Then, we can show that $y$ will be a hole for a number of times that increases with $k$. 
This is because each time $x$ fires, conditioning on $x$ sending exactly one particle to $y$, with equal probability this particle is an oil or a water. 
So the number of unpaired particles at $y$ behaves as a simple random walk on $\mathbb{N}$, reflected at the origin, which is recurrent.
Developing this argument we will obtain that a very large number of holes will be created during this process. At those times, the number of oil-water pairs behaves as a supermartingale. 
Hence, it decreases quickly.

In order to implement this strategy, we need to control the evolution of the locations of the oil-water pairs. The challenge is that  
they behave as a mix of simple random walk, critical branching random walk and subcritical branching random walk, 
depending on (and affecting) the environment of the unpaired particles.
We are able to control this by defining a suitable martingale, which depends on the configuration of the particles. This martingale allows us to relate 
the expected number of oil-water pairs to the Green's function of simple random walk on $G$. This step, which is at the core of our proof, is given in Lemma~\ref{lemma:BRW}; see also Remark~\ref{Remark:supermartingale}.

\section{Graphical representation and properties}
\label{sect:graphical}
In this section we introduce a graphical representation
for the model. 
Via this representation we can prove a 0-1 law for the probability of fixation, and the  Abelian property, where the latter informally states that the number of firings 
%(typically called \textit{odometer function}) 
at a given vertex does not depend on the temporal order of firings of the system and was proved in \cite{BondLevine}.
The structure of this section is inspired by \cite{RollaSidoravicius},
where the authors prove a 0-1 law for two models which are strictly related to the present one, namely \emph{stochastic sandpiles} and  \emph{activated random walks}.

\paragraph{Notation.}
 The graph $G$ is infinite, vertex-transitive with finite degree, and it is fixed along the whole proof.
We fix an arbitrary reference vertex and call it \textit{origin} $ o \in V(G)$.
When considering two vertices $x,y \in V(G)$, we denote by $d(x,y)$ the graph distance between $x$ and $y$, namely the length of the shortest path from $x$ to $y$. 
As a shorthand we also write $x \sim y$ when $d(x,y)=1$.

%, and we denote the set of nearest neighbors of $x$ by
%$\mathcal{N}_x = \{y  \in V(G)\, \, : \, y \sim x \}$.
%Let $B_M := \{x \in V(G) \, \, : \, \, d(o, x) < M\}$  be the ball of radius $M$.
%We let $\d$ denote the degree of the graph,
%and we write $\d_x$ for the degree of $x \in V(G)$ when we want to emphasize that our analysis holds on graphs which are more general than vertex-transitive.

\subsection{Definitions}
%The set of configurations is $\Omega = \mathbb{N}^{V(G)} \times \mathbb{N}^{V(G)} $, and each element of $\Omega$ has the form $\eta = (\eta^o, \eta^w)$, where $\eta^o(x) \in \mathbb{N}$  (resp.\
%$\eta^w(x) \in \mathbb{N}$)
%denotes the number of oils  (resp.\ waters)  
% at $x \in {V(G)}$.
%%Given a particle configuration $\eta \in \Omega$, we let 
%%$\xi = \xi(\eta) \in \mathbb{Z}^{V}$ be the difference between oils and waters, more precisely
%%$\xi(x) : = \eta^o(x) - \eta^w(x)$, for any $x \in {V(G)}$.
%Given a configuration $\eta \in \Omega$, a vertex $x \in {V(G)}$ is called \textit{stable} if $\eta^o(x) \wedge \eta^w(x) = 0$ and it is called  \textit{unstable} otherwise.
%
%\vspace{0.3cm}
%
%\noindent 
%The set of configurations is $\Omega = \mathbb{N}^{V(G)} \times \mathbb{N}^{V(G)} $, and each element of $\Omega$ has the form $\eta = (\eta^o, \eta^w)$, where $\eta^o(x) \in \mathbb{N}$  (resp.\
%$\eta^w(x) \in \mathbb{N}$)
%denotes the number of oils  (resp.\ waters)  
% at $x \in {V(G)}$.
%%Given a particle configuration $\eta \in \Omega$, we let 
%%$\xi = \xi(\eta) \in \mathbb{Z}^{V}$ be the difference between oils and waters, more precisely
%%$\xi(x) : = \eta^o(x) - \eta^w(x)$, for any $x \in {V(G)}$.

The space of possible configurations will be denoted by $\Omega := \mathbb{N}^{V(G)} \times \mathbb{N}^{V(G)}$.
We shall denote an element of $\Omega$ by
\[
 \eta = \big ( \eta^o(x), \eta^w(x)   \big )_{x \in V(G)},
\]
where $\eta^o(x)$  (resp., $\eta^w(x)$) corresponds to the number of oils (resp., waters) at $x$.
%%Let also $\mu = \mu(\nu)$ be the expected number of particles at each vertex, that is,
%%$$
%%%\mu := \nu  \big ( \,  \eta^o(o) + \eta^w(o) \, \big ),
%%\mu := \E_{\nu}  \big ( \,  \eta^o(o) + \eta^w(o) \, \big ),
%%$$
%%where $o \in V(G)$ denotes a reference vertex that we call the origin, and $\E_{\nu}$ denotes the expectation with respect to the density $\nu$.
%Let also $\mu>0$ be the expected number of particles of each type (and thus expected number of pairs) at each vertex, that is,
%$$
%%\mu := \nu  \big ( \,  \eta^o(o) + \eta^w(o) \, \big ),
%%\mu := \E  \big ( \,  \eta^o(o) + \eta^w(o) \, \big ),
%\mu := \E   \eta^o(o) = \E \eta^w(o) ,
%$$
%where $o \in V(G)$ denotes a reference vertex that we call the origin, and $\E$ denotes the expectation with respect to the initial Poisson density of particles.
Also, recall that $\mu>0$ is the expected number of particles at each site in the starting configuration, that is
\[
\mu=\E  \big ( \,  \eta^o(o) + \eta^w(o) \, \big ),
\]
where $o \in V(G)$ denotes a reference vertex that we call the origin.
When investigating the long-time behavior of this model we might expect two possible outcomes, which can depend on $\mu$ and on the properties of the graph $G$, namely \emph{fixation} or \emph{activity}, which we describe below.
For all $x\in V(G)$ and all $t\geq 0$, let $u_t(x) $ denote the number of firings occurred at $x$ by time $t$; we say that the process \emph{fixates} when for any finite set of vertices $A\subset V(G)$ there is a (random) time $\tau_A <\infty $ for which
\[
\forall x\in A, \text{ for all }t>\tau_A \quad u_t(x)= u_{\tau_A}(x).
\]
In other words, no vertex of $A$ fires after time $\tau_A$.
On the other hand, we say that the process is \emph{active} if it does not fixate.
% Our main result states that the process fixates almost surely independently of the initial particle density and the structure of the graph, as long as the initial distribution $\nu$ is ergodic and the underlying graph is vertex-transitive with finite degree.

Given a configuration $\eta \in \Omega$, a vertex $x \in {V(G)}$ is called \textit{stable} if $\eta^o(x) \wedge \eta^w(x) = 0$ and it is called  \textit{unstable} otherwise.

For any $x \in {V(G)}$ and any pair of vertices $y_o, y_w \sim x$, we define a pair of instructions $(\tau^o_{x,y_o}, \tau^w_{x,y_w})$ as an operator acting on configurations $\eta = (\eta^o, \eta^w) \in \Omega$ which are unstable at $x$.
Given such a configuration $\eta$ as input, the operator returns a configuration $\eta_1 = (\eta^o_1, \eta^w_1) \in \Omega$ such that,
for $q \in \{o, w\}$,
$$
\eta^q_1(z) :  = 
\left \{
\begin{array}{ll}
\eta^q(z) - 1 & \mbox{ if $z = x$,} \\
\eta^q(z) +1 & \mbox{ if $z = y_q$,} \\
\eta^q(z)  & \mbox{ otherwise}.
\end{array}
\right .
$$
In words, the operator  $(\tau^o_{x,y_o}, \tau^w_{x,y_w})$ makes one oil  jump from $x$ to $y_o$ and one water  jump from $x$ to $y_w$.

Now we fix an \textit{array} $\tau = \{ \tau^{x,j} : \, x \in V(G), \, j \in \mathbb{N}\}$, where each element $\tau^{x,j}$ is a \textit{pair of instructions} of the form $\tau^{x,j} =  (\tau^{x,j, o},\tau^{x,j, w})$; in particular, each such a pair is an element of the set $\{ (\tau^o_{x,y_o}, \tau^w_{x,y_w}) \ : \ y_o \sim x, y_w \sim x \}$. 

We also need to define a function $\mathbf{h} = ( h(x)\, : \,  x \in V(G))$ that counts the number of pairs of instructions used at each vertex. 
Given the counter $\mathbf{h} $, we say that $x$ \textit{fires} (or that we \textit{topple} $x$, borrowing the notation from the abelian sandpiles setting) when we act on the pair $(\eta, \mathbf{h} )$ through an operator $\Phi_x$ which is defined as,
\begin{equation}
\label{eq:Phioperator}
\Phi_x ( \eta, \mathbf{h} ) =
( \tau^{x, h(x) + 1}  \, \eta, \, \mathbf{h}  + \delta_x),
\end{equation}
where $\delta_x(y)=1$ if $y=x$ and $\delta_x(y)=0$  otherwise.
In words, the operator $\Phi_x$ makes one oil and one water  jump from $x$ simultaneously and then it updates the counter $\mathbf{h} $.
The operation $\Phi_x$ is said to be \textit{legal} for $(\eta, \mathbf{h} )$ if $x$ is unstable in $\eta$, otherwise it is \textit{illegal}.

\subsection{Properties}

We now describe the properties of this representation.
%We keep an instruction array $\tau$ fixed.
For a sequence of vertices $\alpha = ( x_1, x_2, \ldots , x_k)$,
we write $\Phi_{\alpha} = \Phi_{x_k} \Phi_{x_{k-1}}
\ldots \Phi_{x_1}$ and we say that $\Phi_{\alpha}$ is
\textit{legal} for $\eta$ if $\Phi_{x_\ell}$
is legal for $\Phi_{(x_{\ell-1}, \ldots, x_1)} (\eta, \boldsymbol{0}) $
for all $\ell \in \{  2, \ldots , k \}$,
where $\boldsymbol{0}$ is the counter which equals zero at every vertex.
Given a particle configuration $\eta \in \Omega$, a legal sequence $\alpha$ and a fixed array of instructions $\tau$, we write $\Phi_{\alpha} \eta \in \Omega$ for the particle configuration of the pair $\Phi_{\alpha} (\eta, \boldsymbol{0}).$ 
In other words, $\Phi_{\alpha} \eta$  is the particle
configuration which is obtained from $\eta$ 
when we topple the vertices according to  the sequence 
$\alpha$.
Let $m_{\alpha} = \{ m_{\alpha}(x) \, : \,x \in  V(G) \}$
be given by
\begin{equation}\label{eq:m-alpha}
m_{\alpha}(x) \, = \, \sum_{\ell} \mathbbm{1}\{x_\ell = x\}, 
\end{equation}
that is the number of times the vertex $x$ appears in the firing sequence $\alpha$.

We write $m_{\alpha} \geq m_{\beta}$ if
$m_{\alpha} (x)  \,  \geq \, m_{\beta} (x)$ for all  $x \in {V(G)}$.
We write $\eta_1   \geq   \eta_2$ if $\eta^q_1 (x) \, \geq \, \eta^q_2(x)$ for  $q \in \{o, w\}$ and $x \in {V(G)}$. We also write $(\eta', \mathbf{h} ') \geq (\eta, \mathbf{h} )$
if $\eta' \geq \eta$ and $\mathbf{h} ' \geq \mathbf{h} $.

Let $\eta, \eta'$ be two configurations, let $x\in V(G)$, let  $\tau$ be an array of instructions, and let $K$ be a finite subset of ${V(G)}$. A configuration $\eta$ is said to be \textit{stable} in $K$
if all the vertices $x \in K$ are stable. We say that a sequence $\alpha$ is contained in $K$
if all its elements are in $K$, and we say that $\alpha$ \textit{stabilizes} $\eta$ in $K$
if $\Phi_\alpha \eta$ is stable in $K$.
The following property was proved by Bond and Levine.

\begin{Lemma}[Abelian Property, \cite{BondLevine}]\label{lemma:AbelianProp}
   Let $K\subset {V(G)}$ be a finite set.
   If $\alpha$ and $\beta$ are both legal sequences for $\eta$
   that are contained in $K$ and stabilize $\eta$ in $K$, 
   then $m_{\alpha} = m_{\beta}$ and $\Phi_{\alpha} \eta = \Phi_{\beta} \eta$.
\end{Lemma}

For any finite subset $K\subset {V(G)}$, any $x\in {V(G)}$, any particle configuration $\eta$, and any array of instructions $\tau$, we denote by $m_{K,\eta,\tau}(x)$ the number of times that $x$ fires in the stabilization of $K$ starting from  $\eta$ and using the instructions in $\tau$.
Note that by Lemma~\ref{lemma:AbelianProp}, we have that $m_{K,\eta,\tau}$  is well defined.
The following fact is a direct consequence of the Abelian property.

\begin{Lemma}[Monotonicity]\label{lemma:monotonicity}
 For finite subsets  $K \subset K' \subset {V(G)}$ and particle configurations $\eta \leq \eta'$, we have that,  
$$
m_{K, \eta, \tau} \leq m_{K', \eta', \tau}.
   $$   
\end{Lemma}
\begin{proof}
%Fix an array $\tau$, let $\alpha = (x_1, x_2, \ldots, x_k, x_{k+1}, \ldots, x_N)$ be a sequence of firings which stabilizes $\eta$ in $K^{\prime}$ such that the sub-sequence $(x_1, x_2, \ldots, x_k)$ stabilizes $\eta$ in $K$.
Fix an array $\tau$, and let $\beta:=(x_1, x_2, \ldots, x_k)$ be a legal sequence that stabilizes $\eta$ in $K$; then, by Lemma \ref{lemma:AbelianProp} we have that any other legal sequence stabilizing $K$ will use the same number of firings as $\beta$.
By definition, this sequence has not yet stabilized any vertex in the set $K'\setminus K$.
Since the set $K'$ cannot be stable until the set $K$ is stable, the claim follows from \eqref{eq:m-alpha}.
%
%Since for any $n \in \{1, \ldots, N\}$ it follows from \eqref{eq:m-alpha} that $m_{(x_1, x_2, \ldots x_{n+1})} \geq
%m_{(x_1, x_2, \ldots x_{n})} $, the claim follows from Lemma \ref{lemma:AbelianProp}.
\end{proof}
By monotonicity, given any growing sequence of subsets $V_1\subseteq V_2 \subseteq \cdots \subseteq {V(G)}$ such that $\lim_{t\to\infty} V_t={V(G)}$, 
the limit
\begin{equation}\label{eq:lim_Vt}
   m_{\eta, \tau} := \lim\limits_{t\to \infty} m_{V_t, \eta, \tau}
\end{equation}
exists and does not depend on the particular sequence $\{V_t\}_t$.

So far we have fixed a deterministic array $\tau$ and a particle configuration $\eta$. We now introduce a probability measure on the space of instructions and particle configurations.
We denote by $\mathcal{P}$ the probability measure according to which the pairs of instructions
$\tau^{x,j} := (\tau^{x,j,o}, \tau^{x,j,w})$ are independent across different values of $x$, $j$ and $\{o,w\}$, and by $\d_x$ the degree of vertex $x\in V(G)$.
Moreover, the two elements $\tau^{x,j,o}$ and $ \tau^{x,j,w}$ are independent and have distribution
\[
 \mathcal{P} \big ( 
\tau^{x,j, q} =  \tau^q_{x, y_q} \big   )
 := \frac{1}{{\d_x}},
\]
for any  $y_q \sim x$, $q \in \{o,w\}$.
Roughly speaking, under the measure $\mathcal{P}$ the instructions  induce any particle that uses them to perform a step of independent simple random walk.

Finally, we denote by $\mathcal{P}_\nu=\mathcal{P}\otimes \nu$ the joint law of
$\eta$ and $\tau$.
We shall often omit the dependence on $\nu$ by writing $\mathcal{P}$ instead of $\mathcal{P}_\nu$.
The following lemma relates the dynamics of the oil-water model to the stability property of the representation.
Recall that $   \mathbb{P}_{\nu}$ denotes the law of the oil-water dynamics under the assumption that the initial configuration was distributed according to a product of measures $\nu$.
\begin{Lemma}[0-1 law]
   \label{lemma:01law}
%   Let $\nu$ be a translation-invariant, ergodic distribution, and l
Let $m_{\eta, \tau}$ be as in \eqref{eq:lim_Vt}.
   Then 
   \begin{equation}\label{eq:01law}
   \mathbb{P}_{\nu}  (\text{oil and water fixates} ) = \mathcal{P}_{\nu} ( m_{\eta, \tau} (o) < \infty ) \in \{0, 1 \}.
   \end{equation}
\end{Lemma}
%When we average over $\eta$ and $\tau$ using the measure $\mathcal{P}$, 
Lemma \ref{lemma:01law} was proved in \cite{RollaSidoravicius} for two models which are related to oil and water, activated random walk and the stochastic sandpile model. Here we present  the main steps of the proof and we refer to \cite{RollaSidoravicius} for the complete argument.
\begin{proof}[Sketch of the proof of Lemma \ref{lemma:01law}]
The 0-1 law,
\begin{equation}\label{eq:01first}
 \mathcal{P}_{\nu} ( m_{\eta, \tau} (o) < \infty ) \in \{0, 1 \},
\end{equation}
follows from the following fact. 
Conditional on a given initial configuration $\eta$, by connectivity of $G$ and irreducibility of simple random walk, it follows that if $m_{\eta, \tau}(x) = \infty$, then $m_{\eta, \tau}(y)= \infty$ for all $y \in V(G)$.
Since $\mathcal{P}_\nu$ is a product measure and the event $\{m_{\eta, \tau}(y)= \infty$ for all $y \in V(G)\}$ is invariant with respect to any graph automorphism, we deduce  (\ref{eq:01first}).

We now sketch the proof of the identity in (\ref{eq:01law}). 
The proof consists in coupling the quantities $\lim_{t \rightarrow \infty} u_t(x)$ and $m_{\eta, \tau}(x)$.
More precisely, let   $u_{t, M}(x)$ denote the number of firings that occurred at $x$ before time $t$ when no particle is allowed to jump from vertices outside $B_M$, the ball of radius $M$ centered at $o$ (the firings at such vertices are ``frozen'').
The proof consists in two main steps. 

In the first step, one constructs a natural coupling between the variables $m_{B_M, \eta, \tau}(x)$ and $u_{\infty, M}(o) := \lim_{t \rightarrow \infty} u_{t, M}(o)$ as follows. 
%Recall that $\mathbb{P}_{\nu}$ is the law of the oil and water dynamics, which is given by the joint law of the variables  $\eta \in \Omega$ and $\tau$,
%which are independent and distributed 
%according to $\mathcal{P}_{\nu}$,
%and of the sequence of  random variables $\boldsymbol{t} = \{t_{i,x}\}_{i \in \mathbb{N}, x \in V(G)}$, where $\{t_{i,x}\}_{i \in \mathbb{N}}$ are  distributed like the times at which an exponential clock with rate $1$ rings, and the sequences $\{t_{i,x}\}_{i \in \mathbb{N}}$ are independent across the  $x$-s. 
Recall that $\mathcal{P}_{\nu}$ is the joint law of the variables  $\eta \in \Omega$ and $\tau$, under which they are independent.
On the other hand, $\mathbb{P}_{\nu}$ is the law of the oil and water dynamics, given by $\mathcal{P}_{\nu}$ together with the law of the sequence of  random variables $\boldsymbol{t} = \{t_{i,x}\}_{i \in \mathbb{N}, x \in V(G)}$, where $\{t_{i,x}\}_{i \in \mathbb{N}}$ are i.i.d.\ exponential random variables with rate $1$, and the sequences $\{t_{i,x}\}_{i \in \mathbb{N}}$ are independent across $x$. 
The elements of the sequence $\{t_{i,x}\}_{i \in \mathbb{N}}$ represent the times between consecutive attempts for firing $x$.
When such an attempt happens, if $x$ is unstable, one oil and one water perform a simple random walk step from $x$ using the next couple of instructions at $x$ of the array $\tau$.
Thus, by this construction, the random variable $u_{t, M}(x)$ is a deterministic function of the random variables $\boldsymbol{t}$, $\eta$ and   $\tau$. 
Since $u_{t, M}(x)$ is a monotone function in $t$ for every $x$ and every $M$ fixed, the limit
$
u_{\infty, M}(x)  = \lim_{t \rightarrow \infty} u_{t, M}(x)
$
exists.
%Since on a finite set the system fixates within an almost surely finite time and since, by  Lemma \ref{lemma:AbelianProp}
%(Abelian property), $m_{B_M, \eta, \tau} (o)$ does not depend on the order according to which the instructions $\tau$ are used (provided that only legal instructions are used), we deduce from this construction that,
%\begin{equation}\label{eq:coupling}
%\forall r, M \in \mathbb{N},  \quad \mathbb{P}_{\nu} \big ( u_{\infty, M }(o) > r \big ) =
%\mathcal{P}_{\nu} \big ( m_{B_M} (o) > r \big ).
%\end{equation}
Now we observe that on a finite set the system fixates within an almost surely finite time and by  Lemma \ref{lemma:AbelianProp}, $m_{B_M, \eta, \tau} (o)$ does not depend on the order according to which the instructions $\tau$ are used (provided that only legal instructions are used).
Thus, we deduce from this construction that,
\begin{equation}\label{eq:coupling}
\forall r, M \in \mathbb{N},  \quad \mathbb{P}_{\nu} \big ( u_{\infty, M }(o) > r \big ) =
\mathcal{P}_{\nu} \big ( m_{B_M, \eta, \tau} (o) > r \big ).
\end{equation}
The second step of the proof consists in showing that the limits over $M \rightarrow \infty$ and $t \rightarrow \infty$ commute, i.e,
\begin{equation}\label{eq:fact1}
\forall r \in \mathbb{N}, \, \, \,  \, \, \, \, \, 
\mathbb{P}_{\nu} \big (  \lim\limits_{t \rightarrow \infty} \lim\limits_{M \rightarrow \infty} u_{t,M}(o)  > r \,  \big )   =  \mathbb{P}_{\nu} \big ( \,\lim\limits_{M \rightarrow \infty}  \lim\limits_{t \rightarrow \infty} u_{t, M }(o) > r \,  \big ),
 \end{equation}
and that a blow up does not occur in finite time, i.e,
\begin{equation}\label{eq:fact2}
 \forall t \in \mathbb{R}_{\geq 0}, \quad
 \lim\limits_{r \rightarrow \infty}\mathbb{P}_{\nu} \big ( \,  u_{t }(o) > r \,  \big )   = 0.
 \end{equation}
 Equations (\ref{eq:fact1}) and (\ref{eq:fact2})
 and the fact that, 
\begin{equation}\label{eq:fact3}
\forall t \in \mathbb{R}_{\geq 0},  \, \, 
\forall r \in \mathbb{N},\quad
\mathbb{P}_{\nu} \big (  u_t(o) > r  \big )  = 
\lim\limits_{M \rightarrow \infty} \mathbb{P}_{\nu} \big (  u_{t,M}(o) > r  \big ) 
\end{equation}
imply (\ref{eq:01law}).
The proof of (\ref{eq:fact2}) is standard and follows from the fact that, since the jump rates are bounded,  particles starting at an infinite distance from $o$ cannot reach the origin within finite time.
We refer to \cite{RollaSidoravicius} for the proof of 
(\ref{eq:fact1}) given (\ref{eq:coupling}) and of how the equality in (\ref{eq:01law}) follows from these statements.
%We refer to  \cite{RollaSidoravicius2} for the proof of 
%(\ref{eq:fact1}) and (\ref{eq:fact2}) given (\ref{eq:coupling}) and explain how  (\ref{eq:coupling}), (\ref{eq:fact1}),
%(\ref{eq:fact2}), and (\ref{eq:fact3}), imply (\ref{eq:01law}).
%If $\mathcal{P}_{\nu} \big (  m_{B_M, \eta, \tau}(o) < %\infty \big ) = 1$, then,
%by taking the limit $r \rightarrow \infty$,  it follows from (\ref{eq:coupling}), from the fact that $\lim_{M \rightarrow \infty} m_{B_M, \eta, \tau} = m_{\eta, \tau}$, and from (\ref{eq:fact3}) that $u_t(o)$ is almost surely eventually constant.
%This implies that, almost surely,
%for any finite set $ K \subset V(G)$, 
%there exists $ t_K < \infty$
%such that, for all $t > t_K$, and $x \in K$, $ u_t(x)$
%is eventually constant,
%i.e, almost sure fixation.
%Otherwise, by the 0-1 law, $\mathcal{P}_{\nu} \big ( m_{B_M, \eta, \tau}(o) = \infty \big ) = 1$.
%From this, by taking again the limits $r \rightarrow \infty$, we deduce from 
%(\ref{eq:coupling}),
%(\ref{eq:fact1}), from the fact that $\lim_{M \rightarrow %\infty} m_{B_M, \eta, \tau} = m_{\eta, \tau}$ and from (\ref{eq:fact3})
%that $\lim_{t \rightarrow \infty} u_t(o) = \infty$  almost surely. By (\ref{eq:fact2}) we know that
%$\{u_t(o)\}_{t \geq 0}$ cannot blow up in finite time,
%whence, $u_t(o)$ jumps at arbitrarily large times, giving that the system is a.s.\ active.
\end{proof}
From now on, when this is not generating any confusion, we will write $m_{K}$ instead of $m_{K,\eta,\tau}$, and $m(x)$ instead of $m_{\eta, \tau}(x)$ in order to make the paper more readable.

%%%%%%%%%\section{Diffusive fluctuation and number of visits}
%
%
%
%
%
%
%
%
%
%
%
%
%
%
%
%
%

\subsection{Green's function of simple random walk}

In this section we recall some classical facts concerning the simple random walk and we provide some definitions.
We let $X(t)$ denote a simple random walk in $G$, and $P_x$ denote its law when  $X(0)=x \in V(G)$.
We let $E_x$ denote the corresponding expectation. 
Given a set $Z \subset V(G)$ we define 
$
\tau_Z   := \inf\{ t\geq 0  \,  \,   : \,  \,  X(t) \in Z\}$ and 
$\tau^+_Z   := \inf\{ t >0 \, \, : \, \, X(t) \in Z\}.$
If $Z = \{y\}$, we write
$\tau_y$ and $\tau^+_y$ instead of $\tau_{Z}$
and $\tau_Z^+$.
For any $x,y \in V(G)$, we define the Green's function,
\[
G_Z(x,y):= E_x \left [ \sum\limits_{t=0}^{\tau_{Z^c}} \mathbbm{1}\{X(t) = y \} \right ],
\]
where $Z^c := V(G) \setminus Z$.
In words, $G_Z(x,y)$ denotes the expected number of visits to a vertex $y$ performed by a simple random walk started at $x$ and killed upon exiting the set $Z$.

Given a function $g : V(G) \rightarrow \mathbb{R}$, 
$g = (g_x)_{x \in V(G)}$, we let  $\bigtriangleup g : V(G) \rightarrow \mathbb{R}$ denote the discrete Laplacian, that is, for every $x\in V(G)$ we set
$$
   (\bigtriangleup g)_x  := \frac{1}{\d_x}\sum_{y\sim x} (  g_y - g_x ),
$$ 
where we recall that $\d_x$ denotes the degree of $x$.
We say that $g$ is \textit{harmonic} in a set $K \subset V(G)$ if for any $x \in K$, $(\bigtriangleup g)_x = 0$.
The next proposition states some classical facts and its proof can be found, for example, in \cite[Chapter 2]{LyonsPeres}.
\begin{Proposition}\label{prop:classical facts}
Consider a finite set $K \subset V(G)$ and a vertex $y \in K$.
Let $g : V(G) \to \mathbb{R}$ be a function which is harmonic in $K \setminus \{y\}$ and such that $g_y=1$, $g_z=0$ for any $z \in K^c $. 
 Then the function  $g$ is unique  and satisfies
\begin{align}\label{eq:harmonicRW}
 g_w & = P_w( \tau_y < \tau_{K^c}), \quad \forall w \in K,  \\
 - ( \bigtriangleup g)_y    & = 1 - P_y ( \tau^+_y < \tau_{K^c}).
 \end{align}
 Moreover, for all $x,y\in K$ the Green's function satisfies
 \begin{align}
G_K(y,y)  &  =   \frac{1}{ 1 - P_y ( \tau^+_y < \tau_{K^c}) }, \label{eq:G-old} \\ 
G_K(x,y)  &  =  P_x ( \tau_y < \tau_{K^c}) \,  G_K(y,y). \label{eq:G=Ptau}
\end{align}
\end{Proposition}
Now we proceed with the analysis of the function $G_{K}(y,x)$.
\begin{Lemma}\label{lemma:Green}
Suppose that we are given two sets $B$ and $Q$ such that $B \subset Q$, and $o\in Q$. 
Then,
\[
\sum_{y\in B}G_Q(y,o) = G_Q(o,o) \Bigl [\delta_{o\in B}+E_o \bigl [ \# \bigl \{t \leq \tau_{Q^c\cup \{o\}}^+ \ : \ X(t)\in B\setminus \{o\} \bigr \}\bigr ]\Bigr ],
%\text{ until it returns to } Q^c\cup \{o\}\bigr \}\bigr ]\Bigr ],
\]
where
\[
\delta_{o\in B}:=
\left \{
\begin{array}{ll}
1 & \text{ if }o\in B;\\
0 & \text{ if }o\notin B.
\end{array}
\right .
\]
\end{Lemma}
\begin{proof}
For all $y\in B$ by relation \eqref{eq:G=Ptau} we have
\[
G_Q(y,o)=P_y\bigl [ \tau_o<\tau_{Q^c}\bigr ]G_Q(o,o).
\]
Taking the sum over all $y\in B$, reversibility and \eqref{eq:G-old} lead to
\[
\begin{split}
\sum_{y\in B} G_Q(y,o)& =G_Q(o,o)\sum_{y\in B}P_y\bigl [ \tau_o<\tau_{Q^c}\bigr ]\\
&  = G_Q(o,o)\left [\delta_{o\in B}+\sum_{y\in B\setminus \{o\}}P_o\bigl [ \tau_y<\tau^+_{Q^c\cup \{o\}}\bigr ]\frac{1}{ 1 - P_y ( \tau^+_y < \tau_{Q^c\cup \{o\}}) }\right ].
\end{split}
\]
Clearly for all $y\neq o$, 
\[
P_o\bigl [ \tau_y<\tau_{Q^c\cup \{o\}}\bigr ]=E_o \bigl [ \mathbf{1}_{\tau_y<\tau_{Q^c\cup \{o\}}}\bigr ],
\]
and consequently, 
%and taking into account the number of loops from $o$ to $o$ that the random walk makes before reaching $ B$ we obtain
\[
\begin{split}
& G_Q(o,o) \left [\delta_{o\in B}+\sum_{y\in B\setminus \{o\}}  E_o \bigl [ \mathbf{1}_{\tau_y<\tau_{Q^c\cup \{o\}}}\bigr ]\frac{1}{ 1 - P_y ( \tau^+_y < \tau_{Q^c\cup \{o\}}) }\right ]\\
& = G_Q(o,o)\left [\delta_{o\in B}+\sum_{y\in B\setminus \{o\}}E_o \bigl [ \# \bigl \{t \leq \tau_{Q^c\cup \{o\}}^+ \ : \ X(t)=y \bigr \}\bigr ] \right ]\\
& = G_Q(o,o)\Bigl [\delta_{o\in B}+E_o \bigl [ \# \bigl \{t \leq \tau_{Q^c\cup \{o\}}^+ \ : \ X(t)\in B\setminus \{o\} \bigr \}\bigr ] \Bigr ],
\end{split}
\]
concluding the proof.
\end{proof}

\section{Diffusive fluctuations and number of visits}

Now we are able to introduce the following terminology.
\begin{Definition}\label{def:holes and unpaired}
Given a particle configuration $\eta = (\eta^o, \eta^w)$, let $\eta^o(x) \wedge \eta^w(x)$ be the number of \textit{pairs} at $x$, 
and $\eta^o(x) \vee \eta^w(x) - \eta^o(x) \wedge \eta^w(x)$ be the number of \textit{unpaired} particles at $x$. 
We say that $\eta$ has a \textit{hole} at $x$ if the number of oils and waters at $x$ is the same (or, equivalently, if the number of unpaired particles at $x$ is zero).
When we refer to a \textit{pair}, we always refer to two particles of different type.
\end{Definition}

This section is divided into two subsections.
In Section \ref{sect:number holes} we show that, if we assume that the system is active and we stabilize some arbitrarily chosen finite set $K$,  
then at any vertex $x\in K$ which is far enough from the boundary of $K$, we will observe a hole at $x$ \emph{many} times during the stabilization. 
As we pointed out in the proof overview in Section~\ref{sec:overview}, the occurrence of holes is helpful to make the number of oil-water pairs decrease over time.
In Section \ref{sect:BRWS} we introduce a Markov chain which describes an inductive procedure to stabilize $K$ starting from an arbitrary particle configuration.
Such a procedure is  defined in an enlarged probability space where some virtual particles, called \textit{ghosts}, are added to the system whenever a water jumps into a hole. 
We will refer to this procedure as the \textit{ghost-pair stabilization}. The ghosts will play a fundamental role at the end of the proof, in Section~\ref{sect:proof}.

\subsection{Number of waters falling into holes}
\label{sect:number holes}
We start by stabilizing an arbitrary finite set $K \subset V(G)$ following some legal ordering, which we shall determine through a \emph{strategy}.
% and we count the number of times a water falls into a hole. 
A strategy for stabilizing  $K$ is a function $F_K : \Omega \to V(G) \cup \emptyset$ that acts as follows.
Given a particle configuration $\eta$, $F_K$ outputs an arbitrary vertex of $K$ that is currently unstable.
If $\eta$ is stable in $K$ then $F_K(\eta)=\emptyset$.
%More precisely, for any configuration $\eta\in \Omega$ we set
%%\[
%%F_K(\eta):=\emptyset, \text{ if }\forall x\in K \ \eta^o(x)=0 \text{ or }\eta^w(x)=0;
%%\]
%%whereas 
%\[
%F_K(\eta):=
%\left \{
%\begin{array}{ll}
%x, & \text{ arbitrarily chosen among those }x\in K \text{ s.t.\ }\eta^o(x)\geq 1 \text{ and }\eta^w(x)\geq 1,\\
%\emptyset, & \text{ if }\forall x\in K \ \eta^o(x)=0 \text{ or }\eta^w(x)=0.
%\end{array}
%\right .
%\]
%%Note that the vertex $x$ is arbitrarily chosen among those which are currently unstable in $K$.
%%Note that since $K$ is finite, there is a stopping time $T_{F_K}$ that is almost surely finite, such that 
%%
%Later on we will show how this gives us insight into the variable counting the number of times a water particle falls into a hole during the whole stabilization procedure. 

%\begin{Definition}
Let $K \subset V(G)$ be a finite set and $F_K$ a strategy. 
We say that we \textit{stabilize $\eta$ in $K$ following strategy $F_K$} when we perform a  sequence of firings
%, $x_1$, $x_2$, $\ldots$
%%, $x_{T_{F_K}}$
as follows. 
Start by setting $\eta_0 := \eta \in \Omega$ and apply $F_K$ to $\eta_0$. 
If $F_K(\eta_0)=\emptyset$, then we are done as this means that $\eta_0$ is stable.
If $F_K(\eta_0)\neq \emptyset$, then we topple the vertex $F_K(\eta_0)$, and denote by $\eta_1 \in \Omega$ the resulting  configuration. 
If $\eta_1$ is stable then we are done, if not, then we proceed by applying $F_K$ again.
Thus, if $\eta_1$ is unstable, then we proceed to topple $F_K(\eta_1)$, obtaining a new particle configuration which we call $\eta_2 \in \Omega$.
We continue inductively until we reach a random time $T_{F_K}$ at which we have stabilized $K$.
More formally, we set
% until the step  $T_{F_K}$, which is defined as 
$$ 
T_{F_K} : = \inf \big \{ i \in \mathbb{N}_{\geq 0} \, \, : \, \, F(\eta_i) = \emptyset \big  \}.
$$
For any $x \in V(G)$, we define the number of times a water falls into a hole
at $x$ while following the strategy $F_K$ starting from 
the particle configuration $\eta_0$,
\begin{equation}\label{eq:def-nr-ghosts}
H_{K, F_K} (x ) := \Bigl | \{ 0\leq i \leq  T_{F_K}-1 \, \, : \, \, 
\eta^w_i(x) = \eta_i^o(x) \, \, \mbox {and } \, \, 
\eta^w_{i+1}(x) = \eta_{i+1}^o(x) +1
     \}  \Bigr |.
\end{equation}
%\end{Definition}
We emphasize that this variable also depends on $ \eta_0$ and on the chosen array $ \tau$, however we will omit this dependency to simplify the notation.

This procedure defines the sequence of particle configurations $(\eta_i)_{i \in [0, T_{F_K}]}$,
where the last step, $i = T_{F_K}$, is  the  step at which the set $K$ is stable.
In the proof of the next proposition, we will need to introduce some variables which depend also on the instructions $\tau$ which are not ``used'' for the stabilization of the initial particle configuration in $K$. For this reason, we will now define also the steps $i > T_{F_K}$ of the stabilization procedure. 
This will allow to define such variables.
Since the set $K$ is stable at step $i = T_{F_K} $, in order to perform some firings we will need to  add new pairs to the stable configuration, making it unstable.
More precisely, for any step $i > T_{F_K}$, we proceed as follows.
\begin{itemize}
\item If $F_K(\eta_i) = \emptyset$ (i.e,  $\eta_i$ is stable in $K$), then  we add one pair at the origin, obtaining the new particle configuration $\eta_{i+1}$, which is unstable in $K$, and we move to the next step $i+1$.  In this case no vertex fires at step $i$.
\item  If $F_K(\eta_i) \neq \emptyset$ (i.e,  $\eta_i$ is unstable in $K$), then the vertex $F_K(\eta_i) \in K$ fires, and we obtain a new particle configuration $\eta_{i+1}$, which might be stable or unstable in $K$. We move to the step $i+1$.
\end{itemize}
Thus, at any step $i > T_{F_K}$ either one unstable vertex fires or a pair is added at the origin.
In this way the infinite sequence of random variables $(\eta_i)_{i \in \mathbb{N}}$  is well defined.

\begin{Lemma}\label{lemma:number holes}
Assume that the system starting from a particle configuration which is distributed as a product of measure $\nu$ is almost surely active. 
Then, for any $\epsilon >0$ and $M \in \mathbb{N}$, there exists $D  = D(\nu , \epsilon , M) < \infty $ large enough such that,
$$
   \, \inf_{  \substack{ K \subset V(G) : \\ \, d(o,K^c) > D }  } \, 
   \, \inf_{ \substack{ F_K : \Omega \to V(G) \cup  \emptyset  : \\ F_K \mbox{\footnotesize { is a  strategy }  } }  } \,
\,  \mathcal{P}_{\nu} \big (\, \, 
H_{K, F_K} (o)  > M \, \, \big )
  \, \,  \geq  \, \, \, 1 - \epsilon.
$$
\end{Lemma}
Before proceeding to the formal proof we present the main idea behind it, which consists in showing that the value of $H_{K, F_K} (o)$ (defined in \eqref{eq:def-nr-ghosts}) can be associated to the number of visits to zero of a lazy simple random walk on $\Z$.
Once we have established this, classical results give that the number of returns to the origin of a simple random walk on the integers is, with high probability, comparable to the square root of the number of steps performed.
The last step of the proof consists in showing that we can in fact let the walk run for as many steps as we need, in order to deduce the claim.
\begin{proof}[Proof of Lemma \ref{lemma:number holes}]
%%Let $K \subset V$ be a finite set, let $\eta$ be a particle configuration with at least $n$ pairs in $\mathcal{N}_x$, 
%and with no particle at $x$, $x \in K$. Let $(\mathcal{L}(t))_{t \in \mathbb{N}}$ be a lazy random walk on $\mathbb{Z}$ with transition probability
%$$
%P^\mathcal{L} \Big ( \mathcal{L}(t+1) - \mathcal{L}(t) = k %\, \, \bigm| \, \, \tilde{\mathcal{F}_t} \Big )= 
%\begin{cases}
%\frac{1}{deg^2} + ( 1 - \frac{1}{deg})^2 \, \, & \,  \mbox{ %%%%if } k=0 \\
%%\frac{1}{{deg}}(1 - \frac{1}{deg}) \, \, &  \mbox{ if } k= %\pm 1 \\ 
%0 \, \, & \mbox{ otherwise}.
%\end{cases}
%$$  
%for any $t \in \mathbb{N}$, where $P^\mathcal{L}$ is the %law of $\mathcal{L}(t)$ and ${\mathcal{F}_t^{\mathcal{L}}}$ %denotes the filtration which is induced by the first $t$ %steps of the random walk,
%and assume that $\mathcal{L}(0) =0$ a.s.
%We have that
%$$
%\# \{ t \in \mathbb{N} \cap (0, n] \, \, : \, \, %%%%\mathcal{L}(t) = 0\}  \, \,  \prec \, \, H_{K, \eta, \tau}%(x),
%$$
%where $\eta$ and $\tau$ are  $\mathcal{P}_\nu$-distributed.
%
%
%
%Fix the vertex $o\in K$ and consider a random walk defined as follows.
%Every time that $o$ receives a \emph{water} particle, the walk 
To begin, we fix a finite set $K \subset V(G)$ such that $B_D \subset K$, where $B_D$ is the ball of radius $D$ centered at $o$.
Then we stabilize the set $K$ following an arbitrary strategy $F_K$,  as defined before the statement of Lemma \ref{lemma:number holes}.
For any $j \in \mathbb{N}_{>0}$, we let $t_j$ be the $j$-th time a neighbor of the origin fires. 
More precisely, let 
\[
 \mathcal{N}_o:=\{x\in V(G) \ : \ x\sim o\}
 \]
denote the set of neighbors of $o$, and set $t_0 : =0$.
Thus we define for any $j \in \mathbb{N}_{>0}$,
$$
t_j := \inf \{ i > t_{j-1} \, \, : \, \,F_K(\eta_{i-1}) \in \mathcal{N}_o \}.
$$
In words, $t_j$ denotes the first time after $t_{j-1}$ at which a firing occurs at $\mathcal{N}_o$.
We let 
\begin{equation}\label{eq:N_K}
N_K := \sum\limits_{x \in \mathcal{N}_o} m_{K}(x)
\end{equation}
be the number of times that, during the stabilization of $K$, there is a firing from a nearest neighbor of the origin.
We now define a sequence of random variables $\{R_j\}_{j \geq 0}$, which keeps track of the difference between the number of oils and waters at the origin whenever a firing occurs inside $\mathcal{N}_o$.
Subsequently, we show that these random variables are distributed like the steps of a lazy simple random walk on $\mathbb{Z}$. More precisely, first we set
\[
R_0 :=\eta_{t_0}^w(o) -\eta_{t_0}^o (o),
\]
that is, $R_0$ is the difference between the number of waters and the number of oils at vertex $o$ in the initial configuration.
Secondly, for all integers $ j \in \mathbb{N}_{>0}$, we define
\begin{equation}\label{eq:steps}
R_j :=\eta_{t_j}^w(o) -\eta_{t_j}^o (o).
\end{equation}
Let $\d$ denote the degree of any vertex of $G$, which is vertex-transitive.
Since the difference between the number of oils and waters at $o$ can only change when a neighbor of $o$ fires, it immediately follows that the transition probabilities of the walk are given by the following formulas.
The probability to increase of $1$ unit is given by
$$
\mathcal{P}_\nu \left [ R_{j+1}=R_j +1 \mid R_j \right ] = \mathcal{P}_\nu [\eta_{t_{j+1}}^w(o)=\eta_{t_j}^w(o)+1, \ \eta_{t_{j + 1}}^o(o)=\eta_{t_j}^o(o)]
 = \frac{\d-1}{\d^2}.
$$
Symmetrically, we have
\[
\mathcal{P}_\nu \left [ R_{j+1}=R_j -1 \mid R_j \right ] = \mathcal{P}_\nu [\eta_{t_{j+1}}^w(o)=\eta_{t_j}^w(o), \ \eta_{t_{j + 1}}^o(o)=\eta_{t_j}^o(o) + 1]
 = \frac{\d-1}{\d^2},
\]
and finally
\[
\mathcal{P}_\nu \left [ R_{j+1}=R_j \mid R_j \right ] =1-2\frac{(\d-1)}{\d^2}.
\]
At this point, it is clear that $\{R_j\}_{ j \in \mathbb{N}}$ is distributed as the steps of a symmetric lazy random walk on the integers with a given starting value $R_0$.
For any $j \in \mathbb{N}$, let $\mathcal{J}(j)$ be the number of times the random walk jumps from $0$ to $+1$ in the first $j$ steps, i.e,
$$
\mathcal{J}(j) :=  \big | \big \{ k \in [0,j) \, \, : \, \, 
R_k = 0 \mbox{ and }  R_{k+1} = +1 \,\big \}
\big | .
$$ 
By definition, for $N_K$ as in \eqref{eq:N_K} we have that,
\begin{equation}\label{eq:correspondence holes}
H_{K, F_K}(o) = \mathcal{J}(N_K).
\end{equation}
We deduce that, for any $M \in \mathbb{N}$ and $ \varphi \in \mathbb{N}$, 
\begin{align*}
\mathcal{P}_{\nu} \big ( H_{K, F_K}(o) > M \big ) &  \geq 
\mathcal{P}_{\nu} \big ( H_{K, F_K}(o) > M, \,  N_K > \varphi  \big )  \\
& =  \mathcal{P}_{\nu} \big ( \mathcal{J}(N_K) > M, N_K > \varphi \big ) \\
& \geq  \mathcal{P}_{\nu} \big ( \mathcal{J}(\varphi) > M, N_K > \varphi \big ) \\
& \geq 
\mathcal{P}_{\nu} \big ( \mathcal{J}(\varphi) > M \big )
- 
\mathcal{P}_{\nu} \big ( N_K \leq \varphi  \big ) \\
& \geq \mathcal{P}_{\nu} \big ( \mathcal{J}(\varphi) > M \big )
- 
\mathcal{P}_{\nu} \big ( N_{B_D} \leq \varphi  \big ),
\end{align*}
where in the last step we used  the fact that $B_D \subset K$ and  applied Lemma  \ref{lemma:monotonicity}.
Recall that the starting  value  $R_0 = \eta_0^w(o) - \eta_0^o(o)$  is finite almost surely since $\nu$ has finite expectation.
%Recall also that $\mathcal{J}(j)$  corresponds to the number of times such a  random walk jumps from $0$ to $+1$. 
Since the lazy random walk on $\mathbb{Z}$ is recurrent, we deduce 
that for any $\epsilon \in (0, 1)$ and any  $M \in \mathbb{N}$ and any $\nu$ with finite expectation, 
we can choose a value $\varphi = \varphi(\nu, \epsilon, M)$ large enough such that 
$$
\mathcal{P}_{\nu} ( \mathcal{J}(\varphi) > M ) \geq 1 - \frac{\epsilon}{2}.
$$
Since the system is active by assumption, we deduce that there exists $D$ large enough depending on $\epsilon$ and $\varphi$ such that 
$$
\mathcal{P}_{\nu}(N_{B_D}(o) \leq \varphi) \leq \frac{\epsilon}{2},
$$
where, by  Lemma \ref{lemma:AbelianProp} (Abelian property), the previous estimate holds uniformly in the strategy $F_K$.
Combining the previous estimates,  we obtain that 
for any $\epsilon$ and  $M$ we can set $D = D(\nu, \epsilon, M)$ large enough such that, uniformly in $K \supset B_D$ and in the strategy $F_K$,
$$
\mathcal{P}_{\nu} \big ( H_{K, F_K}(o) > M \big ) \geq 1 - \epsilon.
$$
This concludes the proof.
\end{proof}

\subsection{Ghost-pair stabilization}
\label{sect:BRWS}
In this section we define a stabilization procedure where we introduce some auxiliary (virtual) particles, which we will call \emph{ghosts}.
These auxiliary particles do not interact with oils nor waters and perform independent simple random walks.
Each step of the procedure corresponds either to an 
oil-water pair performing a simple random walk step from an unstable vertex, or a ghost performing a simple random walk step and, at any given step of the procedure, at most one ghost is created.
We will refer to this stabilization procedure as  \emph{ghost-pair stabilization}. The procedure is defined in an augmented set of configurations, which we denote by
\[
\widetilde \Omega : =  \mathbb{N}^{ V(G) } \times  \mathbb{N}^{ V(G) }  \times  \mathbb{N}^{ V(G) } , 
\]
where $(\tilde \eta^o, \tilde \eta^w, \tilde \eta^g) \in \widetilde \Omega$ is a triplet such that 
$\tilde \eta^q(x)$ denotes the number oils, waters or ghosts  which are located at $x \in V(G)$ when $q=o$, $q=w$, $q=g$ respectively. 
As before, $\Omega$ will continue to denote the set of configurations of (only) oil and water particles.

\begin{Definition}[Ghost-pair stabilization]	\label{def:ghost-pair}
Let $K \subset V$ be a finite set, let $\sigma \in \Omega$ denote an unstable particle configuration (consisting only of oils and waters, but no ghosts).
%%Recall the terminology which was introduced in Definition \ref{def:ghost-pair}. 
%At time zero,  we start from a configuration
%$\tilde {\eta_0} = (\tilde \eta^o_0, \tilde \eta^w_0, \tilde \eta^g_0) \in \widetilde \Omega$ which is such that oils and waters are placed according to $\sigma$,  $\sigma = (\tilde \eta^o_0,\tilde  \eta^w_0) \in  \Omega$ and, moreover, no ghost is present, i.e, $\tilde \eta_0^g(z) = 0$ for all $z \in V(G)$.
At time zero, we start from a configuration
$\tilde {\eta_0} \in \widetilde \Omega$ such that oils and waters are placed according to $\sigma$, that is $\sigma = (\tilde \eta^o_0,\tilde  \eta^w_0) \in  \Omega$ and, moreover, no ghost is present, i.e., $\tilde \eta_0^g(z) = 0$ for all $z \in V(G)$.
 We let $\boldsymbol{\delta}_x \in \mathbb{N}^{V(G)}$ be the vector which equals one at $x \in V(G)$ and zero everywhere else.
Inductively, for every integer $t \geq 0$, we first follow (i) and then (ii) described below.
\begin{enumerate}[(i)]
\item Either a ghost or an oil-water pair in $\tilde \eta_t$ which are located on a vertex of $K$ perform a simple random walk step, where the latter means that an oil and a water which are located at the same vertex take one independent step according to simple random walk. 
This leads to a new particle configuration which we call $ \theta_t \in \tilde \Omega$.
\item  If during (i) a water falls into a vertex $x \in K$ which is hosting a hole (i.e., $ {\tilde \eta}_t^o(x) = \tilde {\eta}_t^w (x)$ and  $ \theta_t^w(x) = \theta_t^o(x) + 1$), then a ghost is added at that vertex, that is, 
\[
\tilde {\eta}_{t+1}^g := \theta_t^g + \boldsymbol{\delta}_x, \quad \text{ and }\quad \tilde \eta_{t+1}^q := \theta_{t+1}^q, \quad q \in \{o,w\},
\]
otherwise nothing happens,
(i.e,  $\tilde \eta_{t+1} : = \theta_t$).
This defines $\tilde \eta_{t+1}$. 
% Afterwards, we move to the step $t+1$.
\end{enumerate}
Since $K$ is finite, after  an almost surely finite number of steps no pair and no ghost is present in $K$ and the procedure stops.
We define 
\[
T=T(K):=\inf \{s\geq 0 \ : \ K\text{ is stable with respect to }(\tilde \eta^o_s, \tilde \eta^w_s)\text{ and }\tilde{\eta}^g_s(y) =0 \, \, \forall y \in K\},
\]
and for every $t \geq T$ we set $\tilde \eta_t := \tilde \eta_T$.
In the following we set, for any $y \in K$,
\[
\tilde m (y):=\# \{\text{times that either a ghost or an oil-water pair jumps from } y\},
\]
and we denote by  $\tilde P_{K, \sigma}$ the law of the ghost-pair stabilization.
\end{Definition}
%\begin{Remark}\label{Remark:Abelian}
%At any step $t \in \mathbb{N}$, whether a ghost or a oil-water pair perform a simple random walk step and the vertex from which they jump  is determined by a function $F$ which depends on the first $t-1$ steps of the procedure (i.e., $F$ is $\mathcal{F}_{t-1}$-measurable, where $\mathcal{F}_t$ is the sigma-algebra generated by the first $t$ steps of the process). Thus, the evolution of the process depends on such a function  $F$. However, by the Abelian property, the quantity $\tilde  m (y)$ does not depend on such a  function. It is simple to see that the Abelian property holds also when ghosts are added to the system, since they do not interact with oils and waters and perform a simple random walk.
%\end{Remark}
The lemma below is the main step in the proof of our main result. 
It shows that, during the stabilization procedure started from an arbitrary (unstable) configuration $\sigma$, 
the expected value of $\tilde m (y)$, for any fixed $y \in K$, 
can be estimated in terms of the Green's function of simple random walk and of the number of pairs in the initial configuration.
\begin{Lemma}\label{lemma:BRW}
For any finite set $K \subset V$,   any vertex $y \in K$, and  any unstable particle configuration $\sigma :=(\tilde \eta^o_0,\tilde \eta^w_0)\in \Omega$,
\begin{equation}\label{eq:green-fct}
\tilde E_{K, \sigma} \big (  \tilde  m ( y) \big )  =  \sum\limits_{x \in K}  \, \big ( \tilde \eta^o_0(x) \wedge \tilde \eta^w_0(x) \, \big) \, G_K(x,y),
\end{equation}
where $\tilde E_{K, \sigma}$ denotes the expectation with respect to $\tilde P_{K,  \sigma}$.
\end{Lemma}
\begin{proof} %[Proof of Proposition \ref{prop:BRW}]
Let $K \subset V$ be a finite set, fix one vertex $y \in K$. 
Let $g : V(G) \mapsto \mathbb{R}$ be the  function which is harmonic in $K \setminus \{y\}$ and such that $g_y=1$, $g_z=0$ for any $z \in K^c $.
Recall that $\tilde \eta_t$ denotes the state of the process (cf.\ Definition \ref{def:ghost-pair}) at time $t$. 
For convention, we refer to as \emph{step $t$} the transition from $\tilde \eta_{t-1}$ to $\tilde \eta_t$, and let $x_t$ denote the vertex from which a pair or a ghost jumps at step $t$.
%, and note that the choice of $x_t$ is measurable with respect to $\mathcal{F}_{t-1}$.
For each $t\in \N_{\geq 0}$ define 
\begin{equation}
   M_t : = \sum_{x \in K }  \Big ( \tilde \eta_t^o(x) \wedge \tilde \eta_t^w(x)  + \tilde \eta_t^g(x)\Big ) \,g_x
      -
     (\bigtriangleup g)_y \, \, \sum\limits_{i=1}^{t} \mathbbm{1}\{ x_i = y   \}.
     \label{eq:mt}
\end{equation}
% where $y $ is a fixed vertex.
Let $( \tilde \Sigma, \tilde{\mathcal{F}}, \tilde P_{K, \sigma})$ be the probability space where the process $\{\tilde \eta_t\}_t$ is defined; the proof of the proposition will follow from the fact that $M_t$ is a martingale, namely 
\begin{equation}\label{eq:Martingale}
\tilde E_{K, \sigma} [ \, M_t \, \bigm|  \, \mathcal{F}_{t-1}  \, ] = M_{t-1} .
\end{equation}
We will now prove (\ref{eq:Martingale})
considering different cases.

In the first case, consider that at step $t$ a ghost jumps from $x_t = b \in K$.
Then, in this case,
\begin{equation}\label{eq:compt1}
\tilde E_{K, \sigma} \, [ \, M_t \, \bigm|  \, \mathcal{F}_{t-1}  \,  ] =  M_{t-1} \, \,  -
g_b \, \,  + 
\frac{1}{\d_b}
\Big ( \,
\sum\limits_{z \sim b} g_{z}  \, \Big)  - 
\mathbbm{1}  \{b = y \} ( \bigtriangleup g)_y
=M_{t-1},
\end{equation}
where the last identity holds since $g$ is harmonic in $K \setminus \{y\}$.

In the second case, consider that at step $t$ an oil and water pair jumps from some vertex $x_t = b \in K$.
Let $\mathcal{N}_{b,t}^{oe}$ (resp.\ $\mathcal{N}_{b,t}^w$ ) be the set of vertices $z \in V(G)$ such that $z \sim b$
and $\tilde \eta_{t-1}^o(z) - \tilde \eta^w_{t-1}(z) \geq 0$ 
(resp.\  $\tilde \eta_{t-1}^o(z) - \tilde \eta^w_{t-1}(z) <0$).
Note that $\mathcal{N}_{b,t}^{oe}$ and 
$\mathcal{N}_{b,t}^w$ are measurable with respect to $\mathcal{F}_{t-1}$.
Then, denoting by $z_o$ (resp.\ $z_w$) the destination of the oil (resp.\ water) in the next sum,
\begin{equation}\label{eq:compt2}
\begin{split}
\tilde E_{K, \sigma} & [ \, M_t \, \bigm|  \, \mathcal{F}_{t-1}  \,  ]  =   M_{t-1} \, \,  - g_b \, \,  +  \frac{1}{\d_b^2}
\Big ( \,
\sum\limits_{ \substack{ z_w \in \mathcal{N}_{b,t}^{oe}},   z_o \in \mathcal{N}_{b,t}^{oe}} g_{z_w}  \, \Big)  \,  + \\ 
& \, + \frac{1}{\d_b^2}
\Big ( \,
\sum\limits_{ \substack{  z_w \in \mathcal{N}_{b,t}^{oe}, z_o \in \mathcal{N}_{b,t}^{w}}} \big ( g_{z_o} + g_{z_w} \big )  \, \Big )  \, +  
\frac{1}{\d_b^2}
\Big ( \,
\sum\limits_{ \substack{ z_w \in \mathcal{N}_{b,t}^{w}},   z_o \in \mathcal{N}_{b,t}^{w}} g_{z_o}    \, \Big )  \,   -  \, 
\mathbbm{1}  \{ b = y \} ( \bigtriangleup g)_y \\
& =
M_{t-1} \, \,  - \, \, 
g_b \, \,  + 
\frac{| \mathcal{N}_{b,t}^w|+| \mathcal{N}_{b,t}^{oe}|}{\d_b^2} \Big ( \,
\sum\limits_{z \sim b} g_{z} \,  \Big )  \, 
 -  \, 
\mathbbm{1}  \{ b = y \} ( \bigtriangleup g)_y \\
& = M_{t-1} \, \,  -
g_b \, \,  + 
\frac{1}{\d_b}
\Big ( \,
\sum\limits_{z \sim b} g_{z}  \, \Big)  - 
\mathbbm{1}  \{ b = y \} ( \bigtriangleup g)_y \\
& = M_{t-1},
\end{split}
\end{equation}
 where the last identity follows from the fact that $g$ is harmonic in $K \setminus \{y\}$.
This concludes the proof of (\ref{eq:Martingale}).

Now we prove the lemma  using (\ref{eq:Martingale}).
Recall that $T$ is the first time at which the set $K$ is stable and no ghost is present in $K$.
Since $K$ is finite, $\E T<\infty$ almost surely, furthermore $M_t$ has bounded increments, thus,
%For any integer $n \in \mathbb{N}$, define $T_n := T \wedge n$. Since $T_n$ is an almost surely bounded stopping time, 
the conditions of the optional stopping theorem are fulfilled and we deduce that
$$
\tilde E_{K,  \sigma} \, [ \, M_{T} \, ] 
= \tilde E_{K, \sigma} \, [ \, M_{0} \, ] = \sum\limits_{x \in K} \big (   \tilde \eta^o_0(x) \wedge \tilde \eta^w_0(x) \big )\, g_x,
$$
recalling that $\tilde \eta^o_0 \wedge \tilde \eta^w_0$ corresponds to the number of pairs at $x$ in the initial configuration $\sigma$ and that we start with no ghost at time zero.
%By taking the limit $n\rightarrow \infty$, from the Monotone convergence theorem we deduce that,
%$$
%\tilde E_{K, \sigma} \, [ \, M_{T} \, ] 
%=  \sum\limits_{x \in K} g_x \, \, \big ( \tilde \eta^o_0(x) \wedge \tilde \eta^w_0(x) \big ).
%$$
This leads to, 
\begin{equation}\label{eq:finalequation}
- ( \bigtriangleup g)_y  \, 
\tilde E_{K, \sigma} \big ( \, \tilde m(y) \, \big ) =
- ( \bigtriangleup g)_y  \, 
\tilde E_{K, \sigma} \Big ( \sum\limits_{t=1}^{\infty}
\mathbbm{1} \{x_t = y \}  \Big ) = \sum\limits_{x \in K} \big (  \tilde \eta^o_0(x) \wedge \tilde \eta^w_0(x) \big )\, g_x.
\end{equation}
Using Proposition \ref{prop:classical facts}, we obtain %(\ref{eq:green-fct}) and conclude the proof.
\begin{eqnarray*}
   \tilde E_{K, \sigma} \big ( \, \tilde m(y) \, \big ) 
   &=& \sum\limits_{x \in K} \big (  \tilde \eta^o_0(x) \wedge \tilde \eta^w_0(x) \big )\, g_x G_K(y,y)\\
   &=& \sum\limits_{x \in K} \big (  \tilde \eta^o_0(x) \wedge \tilde \eta^w_0(x) \big )\, P_x(\tau_y < \tau_{K^c}) G_K(y,y)\\
   &=& \sum\limits_{x \in K} \big (  \tilde \eta^o_0(x) \wedge \tilde \eta^w_0(x) \big )\, G_K(x,y).
\end{eqnarray*}
\end{proof}

\begin{Remark}\label{Remark:supermartingale}
In the overview in Section~\ref{sec:overview}, we noticed that the oil-water pairs move as a mix of simple random walk, critical branching random walk, and subcritical branching random walk, depending on the environment. 
In particular, the total number of pairs which are present in the oil and water system (with no introduction of ghosts) is a super-martingale. 
In fact, if we fire a vertex that does not neighbor a hole, 
then the number of oil-water pairs behaves as a martingale; otherwise, the expected number of pairs strictly decreases.
It is extremely hard to control the evolution of the system consisting exclusively of oil-water pairs, 
because this requires controlling the evolution of the configuration of holes and of pairs at the same time, which are strongly correlated.
The introduction of ghosts compensates the pairs that are lost when we fire a vertex neighboring a hole. 
In particular, if we were to define $M_t$ as simply $\sum_{x \in K } ( \tilde \eta_t^o(x) \wedge \tilde \eta_t^w(x)  + \tilde \eta_t^g(x))$, we would be able to show that $M_t$ is a super-martingale (where 
it would not be a martingale only due to particles or ghosts jumping out of $K$).
The introduction in $M_t$ of the function $g$, which is harmonic everywhere in $K$ but at $y$, is to make each firing at $y$ give an extra contribution.
This allowed us to add the negative term at the end of~\eqref{eq:mt}, which counts the number of times that a pair or a ghost jumps from $y$; that is, it allows us to estimate $\tilde m(y)$.
Both ghosts and pairs contribute to the total number of jumps $\tilde m(y)$, and to show fixation we actually need to control only the contribution given by oil-water pairs. 
In Section \ref{sect:proof}, we will isolate the two contributions and compare them.
\end{Remark}

\section{Proof of Theorem \ref{thm:fixation}}\label{sect:proof}
In this section we present the proof of our main theorem, which works by contradiction and uses the ghost-pair stabilization (recall Definition \ref{def:ghost-pair}).
As explained in Section \ref{sect:BRWS}, the expected number of pairs which are present in the system when a firing occurs at a nearest neighbor of a hole is strictly decreasing.
Ghosts are introduced to compensate the loss of pairs,
in such a way that the  total number of pairs and ghosts which are present at any step of the ghost-pair stabilization is a martingale.
The proof of the theorem is based on the following idea. Suppose the system is active. Then, Lemma \ref{lemma:number holes} implies that a \emph{large} number of ghosts is produced at most vertices; but  ghosts are produced to compensate the decrease in the number of pairs.
Thus if many ghosts are produced, that means that a \emph{large} number of pairs was lost.  
The proof consists in showing that it is not possible to produce so many ghosts if we start with a finite density of pairs, leading to the desired contradiction. 
To show this fact we will exploit the Green's function of a suitably defined random walk to relate the expectation of three different quantities, namely the number of particles which start from every  vertex, the number of ghosts which are produced at every  vertex and the number of times a ghost or a pair visit the origin. 
% Considering only the expectation of these variables will allow to  overcome the hard problem of dealing with their correlations.

To begin, we state an auxiliary result.
From now on, fix an arbitrary sequence of finite sets, namely the sequence of balls centered at the origin and of radius $L\geq 1$, which we denote by $\{B_L\}_{L\in \N}$.
% that satisfy the following constraints:
%\begin{itemize}
%\item[(i)] For all $ L \geq 1$ we have $S_ L \subset S_{L+1}$;
%\item[(ii)] The origin $o\in S_1$;
%\item[(iii)] For all $ L \geq 1$, $S_L \supset B_L$ where once again $B_L$ denotes the ball (in the graph metric) of radius $L$ centered at $o$.
%\textcolor{red}{If $S_L = B_L$, aren't these conditions automatically fulfilled? If yes, why we need to define $S_L$? }
%\end{itemize}
\begin{Lemma}\label{lem:PropertiesGreen}
For any $D \in \mathbb{N}$ there exists $L_0 = L_0(D)$ large enough such that, for any $L > L_0$,
$$
\sum\limits_{x \in B_L} G_{B_L}(x,o) < 10 \sum\limits_{ \substack{ x \in B_L :  \\ B(x,D) \subset B_L}} G_{B_L}(x,o),
$$
where $B_L^c := V(G) \setminus B_L$, and $B(x, D)$ is the ball of radius $D$ centered at $x$.
\end{Lemma}
We will now prove Theorem \ref{thm:fixation} using Lemma \ref{lem:PropertiesGreen}. The proof of Lemma \ref{lem:PropertiesGreen} will be presented afterwards.
\begin{proof}[\textbf{Proof of Theorem \ref{thm:fixation}}]
%Let $S_L$ be one of the sets of the sequence $\{S_L\}_{L \in \mathbb{N}}$ which was defined above.
%Recall the terminology which was introduced in Definition \ref{def:holes and unpaired}, recall the ghost-pair stabilization, which was defined in Section \ref{sect:BRWS}.
To begin, for any $L$ fixed and arbitrarily large, consider the following procedure.
Stabilize the set $B_L$ following the ghost-pair stabilization:  while stabilizing the set $B_L$, every time a water falls into a hole, a ghost is created at that vertex. 
Ghosts perform independent simple random walks until they leave $B_L$.
For any $x\in B_L$ we define,
\[
\begin{split}
& \CorG_L(x)  :=\text{number of \emph{pairs} or \emph{ghosts} that jump from }x\text{ during the stabilization of }B_L,\\
& m_L(x)  :=\text{number of firings at }x\text{ during the stabilization of }B_L,\\
& \GVis_L(x)  :=\text{number of \emph{ghosts} that jump from }x\text{ during the stabilization of }B_L,\\
& \GSt_L(x)  :=\text{number of ghosts \emph{started} (created) at }x\text{ during the stabilization of }B_L.
\end{split}
\]
Recall that $\mu = \mu (\nu) \in (0, \infty)$ is the expected number of particles which are present at each vertex in the starting configuration.
 We claim that, for any $L \in \mathbb{N}$,
\begin{align}\label{eq:EwL}
\tilde \E_{\nu}   \bigl ( \CorG_L(x)\bigr ) & \leq \sum_{y\in B_L}\mu \, G_{B_L}(y,x); \\
\label{eq:EgL}
\tilde \E_{\nu} \left  ( \GVis_L(x) \right  ) &  = \sum_{y\in B_L} \tilde \E_{\nu} \bigl ( \GSt_L(y)\bigr )G_{B_L}(y,x),
\end{align}
where $\tilde \E_{\nu}$ denotes the expectation of the measure which is defined in the enlarged probability space of oils, waters and ghosts.
Equation (\ref{eq:EwL}) follows from Lemma \ref{lemma:BRW} by averaging over the initial particle configuration and observing that the expected number of pairs of the initial configuration at every vertex cannot be larger than the expected number of particles.
Equation  (\ref{eq:EgL}) follows from linearity of expectation and from the fact that every ghost performs an independent simple random walk until it leaves $B_L$.
We also claim that, if we assume that the system starting with initial particle distribution $\nu$ is almost surely active, then there is a large enough $D = D(\nu)$ such that for any $L > D$, and for any $x \in B_L$ such that $ B(x, D) \subset B_L$,
\begin{equation}
\label{eq:3/4mu2}
\tilde \E_{\nu} \bigl [ \GSt_{L}(x) \bigr ]  \geq 10 \mu.
\end{equation}
%where $\mu = \mu(\nu)$ is the density of particles of $\nu$.
Indeed, equation (\ref{eq:3/4mu2}) follows from Lemma
\ref{lemma:number holes} and from the fact that $G$ is vertex-transitive, since, by definition, a ghost is produced at $x$ every time a water falls into a hole and the estimate in Lemma \ref{lemma:number holes} holds uniformly over all strategies.

For the rest of the proof, we will keep assuming that the system is almost surely active and we will look for a contradiction. We will also keep the value $D$ fixed as above.
By definition, $m_L(x)$ is the number of times that a pair jumps from $x$, and this number equals the number of times that a ghost or a pair jump from $x$ minus the number of times a ghost jumps from $x$, that is,
\[
m_L(x)=\CorG_L(x)-\GVis_L(x).
\]
It follows from the linearity of expectation and from 
(\ref{eq:EwL}), (\ref{eq:EgL}), and (\ref{eq:3/4mu2}), that
\begin{align}\label{eq:EmL}
\tilde \E_{\nu} \bigl ( m_L(x)\bigr ) &  \leq   \sum_{y\in B_L}\mu G_{B_L}(y,x)-\sum_{y\in B_L} \tilde \E_{\nu} \bigl ( \GSt_L(y)\bigr )G_{B_L}(y,x) \\  \nonumber
& \leq 
\sum_{y\in B_L}\mu G_{B_L}(y,x) - \sum_{ \substack{ y\in B_L : \\  B(y,D) \subset B_L}} \tilde \E_{\nu} \bigl ( \GSt_L(y)\bigr )G_{B_L}(y,x) \\  \label{eq:towards-prop}
& \stackrel{\eqref{eq:3/4mu2} }{\leq } \mu \, \, \Big ( 
\sum\limits_{y \in B_L} G_{B_L}(y,x) - 10  \sum_{ \substack{ y\in B_L : \\  B(y,D) \subset B_L}} G_{B_L}(y,x)
\Big ) .
\end{align}
From Lemma \ref{lem:PropertiesGreen}, we conclude that $\tilde \E_{\nu} \bigl ( m_L(o)\bigr ) < 0 $ for large enough $L$. 
Since the number of firings cannot be negative, the above leads to the desired contradiction. We conclude that the probability that the system is active is strictly smaller than 1. 
By Lemma \ref{lemma:01law} (the 0-1 law), we deduce that the system fixates almost surely, concluding the proof.
\end{proof}

It remains to prove Lemma \ref{lem:PropertiesGreen}.

\begin{proof}[Proof of Lemma \ref{lem:PropertiesGreen}]
% Fix a value $\r>0$ large enough and let $D$ be the constant in Lemma \ref{lemma:number holes}.
Pick $L_0$ very large such that %relation \eqref{eq:3/4mu2} holds for all $L\geq L_0-D$ and 
\begin{equation}\label{eq:choiceL}
L_0\geq D\left ( 1+\d^D \right ).
\end{equation}
For $L\geq L_0$, consider the set $B_L$ and on each vertex on the external boundary of $B_L$ (denoted by $\partial B_L$) place a ball of radius $D$, and define the annulus
\[
A_{L,D}:=\bigcup_{x \in  \partial B_L}B(x,D).
\]
Note that by construction it follows that
\begin{equation}\label{eq:Binclusion}
B_{L-D}\subset \bigl (B_L\setminus A_{L,D}\bigr ).
\end{equation}
%By using Proposition \ref{prop:number holes} together with our choice of $\r\geq D$ and relation \eqref{eq:3/4mu2}, we obtain that the above is at most
To establish the lemma, it suffices to show that 
\[
\sum_{y\in  B_L\setminus A_{L,D}}-9 G_{B_L}(y,o)+\sum_{y\in A_{L,D}\cap B_L} G_{B_L}(y,o)<0.
\]
Now we will apply Lemma \ref{lemma:Green}, and as a shorthand for all sets $B \subset Q$ we set
\[
\begin{split}
& E_o \mathcal{R} (B, Q^c) := 
E_o \bigl [ \# \text{steps of random walk in }B\setminus \{o\}\text{ before returning to } Q^c\cup \{o\}\bigr ],
\end{split}
\]
that is, the expected ``range'' (from which the symbol $\mathcal{R}$) made by a random walk started at $o$ inside the set $B\setminus \{o\}$ before exiting $Q$ or reaching $\{o\}$.
% As a shorthand we set
% \[
% E_o \mathcal{R} (B) := E_o \mathcal{R} (B, B^c),
% \]
% that is, the expected number of steps of random walk in $B\setminus \{o\}$  before returning to $ B^c\cup \{o\}$.
Therefore we need to show
\[
- 9 \left  [\delta_{o\in B_L\setminus \, A_{L,D}}+E_o \mathcal{R} \Bigl (B_L\setminus A_{L,D}, B_L^c \Bigr )\right ] +\Bigl [\delta_{o\in A_{L,D}\cap B_L}+E_o \mathcal{R} \bigl (A_{L,D}\cap B_L, \, B_L^c \bigr )\Bigr ]<0.
\]
We start by observing that by \eqref{eq:Binclusion} we have that $\delta_{o\in B_L\setminus A_{L,D}}=1$ and $\delta_{o\in A_{L,D}\cap B_L}=0$.
As a consequence, the above is equivalent to 
\begin{equation}\label{eq:E-bound}
-9 \Bigl [1+E_o \mathcal{R} \Bigl (B_L\setminus A_{L,D}, B_L^c \Bigr )\Bigr ]  + E_o \mathcal{R} \bigl (A_{L,D}\cap B_L, \, B_L^c \bigr )<0.
\end{equation}
%From the definition it follows that
%\[
%1+\E_o \mathcal{R} \Bigl (B_L\setminus \bigl (A_{L,\r} \cup \{o\}\bigr )\Bigr )\geq 0,
%\]
%and 
%\begin{equation}\label{eq:E-G}
%\E_o \mathcal{R} \bigl (A_{L,\r}\cap B_L, \, B_L^c \bigr )\geq 0.
%\end{equation}
Thus, at this point the only thing that remains to be shown is that for all $L\geq L_0$ we have
\begin{equation}\label{eq:TBS}
E_o \mathcal{R} \Bigl (B_L\setminus A_{L,D},B_L^c \Bigr )\geq  E_o \mathcal{R} \bigl (A_{L,D}\cap B_L, \, B_L^c \bigr ).
\end{equation}
% In fact, suppose that relation \eqref{eq:TBS} holds, then   \eqref{eq:E-bound} would be verified as follows
% \[
% \begin{split}
% -9 & \Bigl [1+E_o \mathcal{R} \Bigl (B_L\setminus A_{L,D} \Bigr )\Bigr ]  + E_o \mathcal{R} \bigl (A_{L,D}\cap B_L, \, B_L^c \bigr )\\
% & \stackrel{\eqref{eq:E-bound} }{\leq }-9  \Bigl [1+E_o \mathcal{R} \bigl (A_{L,D}\cap B_L\bigr )\Bigr ] + E_o \mathcal{R} \bigl (A_{L,D}\cap B_L\bigr ) \\
% & =  -9-8 \Bigl [E_o \mathcal{R} \bigl (A_{L,D}\cap B_L \bigr )\Bigr ]  <0.
% \end{split}
% \]
In order to show \eqref{eq:TBS} we start by observing that \eqref{eq:Binclusion} gives
\begin{equation}\label{eq:L-r}
E_o \mathcal{R} \Bigl (B_L\setminus A_{L,D}, B_L^c\Bigr ) \geq (L-D)P_o \bigl (\tau_{ (B_L\setminus A_{L,D})^c}<\tau_o^+ \bigr ).
\end{equation}
This is in fact a crude lower bound on the expected number of steps that a simple random walk started at $o$ has to make in order to exit the set $B_{L-D}\subset B_L\setminus A_{L,D} $, but it will be enough for our purposes.
%This lower bound can be deduced in the following way.
%Suppose that the random walk is at some vertex $z \in B_{L-\r}$, then the probability that it will exit $B_L\setminus A_{L,\r}$ before returning to the origin is given by 
%\[
%P_z \bigl (\tau_{ (B_L\setminus A_{L,\r})^c}<\tau_o^+ \bigr ).
%\]
%Now, since the origin is the vertex that is the furthest away from the boundary of the ball $B_{L-\r}$, we have that uniformly over all $z \in B_{L-\r}$
%\[
%P_z \bigl (\tau_{ (B_L\setminus A_{L,\r})^c}<\tau_o^+ \bigr ) \geq P_o \bigl (\tau_{ (B_L\setminus A_{L,\r})^c}<\tau_o^+ \bigr ) .
%\]
%Furthermore, it is clear that a walk starting at $o$ will have to cross at least $L-\r$ edges in order to exit $B_{L-\r}\subset B_L\setminus A_{L,\r} $, from which relation \eqref{eq:L-r} follows.
In order to obtain \eqref{eq:TBS}, we will show that the following holds:
\begin{equation}\label{eq:Delta-r}
E_o \mathcal{R} \bigl (A_{L,D}\cap B_L, B_L^c \bigr )\leq D \d^D P_o \bigl (\tau_{ (B_L\setminus A_{L,D})^c}<\tau_o^+ \bigr )  .
\end{equation}
A way to see why the above is true is the following.
% The expected number of steps of a nearest neighbor random walk inside  $A_{L,D}\cap B_L$ before exiting $B_L$ is simply the expected time spent inside $A_{L,D}\cap B_L$ before exiting $B_L$.
By construction, every vertex $x\in A_{L,D}\cap B_L$ is at some distance $ r\leq D$ from $ B_L^c$. 
Hence, whenever a random walk starts at $x\in A_{L,D}\cap B_L$ with $d(x,  B_L^c)=r$, it has probability bounded from below by $1/\d^r\geq 1/\d^D$
%\[
%%\frac{1}{\d^\r}
%1/\d^\r
%\]
of exiting $B_L$.
(This is clear since the above is a lower bound on the probability of taking  $D$ steps in the same direction to reach $B_L^c$.)
In particular, the random variable representing the time needed by the random walk started at $x$ to exit $B_L$ is stochastically dominated by a geometric random variable with success parameter $\d^{-D}$.
%Therefore, 
%%the expected number of steps needed by a walk started at $x$ at distance $r$ from $B_L$ to exit $B_L$ is at most 
%\[
%E_x \bigl [ \# \text{ steps needed by random walk with }d(x,  B_L^c)=r\text{ to exit }B_L\bigr ]\leq \d^\r,
%\]
%uniformly over all $r\leq \r$.
%By taking the union bound over all $x\in A_{L,\r}\cap B_L$, and thus over all possible values of $r$ we obtain
%\[
%\begin{split}
%& E \bigl [ \# \text{ steps needed by random walk started in } A_{L,\r}\cap B_L\text{ to exit }B_L\bigr ]\\
%& \leq \sum_{r\leq \r}E_x \bigl [ \# \text{steps needed by random walk with }d(x,  B_L^c)=r\text{ to exit }B_L\bigr ]\leq \r\d^\r,
%\end{split}
%\]
%which implies \eqref{eq:Delta-r}.
If after $D$ steps the random walk has not exited $B_L$, then we just iterate the above, which implies \eqref{eq:Delta-r}.
Using the fact that $L\geq L_0$, where $L_0$ can be chosen as in \eqref{eq:choiceL} 
%we obtain
%\[
%(L-\r)P_o \bigl (\tau_{ (B_L\setminus A_{L,\r})^c}<\tau_o^+ \bigr )\geq \r \d^\r P_o \bigl (\tau_{ (B_L\setminus A_{L,\r})^c}<\tau_o^+ \bigr ),
%\]
%finally proving 
establishes \eqref{eq:TBS}, which concludes the proof of Lemma \ref{lem:PropertiesGreen}.
\end{proof}

\section*{Acknowledgements}
This work started when E.\ Candellero was affiliated to the University of Warwick and L.\ Taggi was affiliated to the Technische Universit\"at Darmstadt.
A.\ Stauffer and L.\ Taggi acknowledge support from EPSRC Early Career Fellowship EP/N004566/1,
L.\ Taggi acknowledges support from DFG  German Research Foundation BE 5267/1.

% \nocite{}
\bibliographystyle{amsalpha}
\bibliography{Bibliography}

\end{document}